\documentclass{amsart}
\usepackage{amssymb,latexsym}
\usepackage{amscd,amsthm}

\usepackage[all]{xy}
\usepackage{tikz}

\usepackage{tikz-cd}

\newtheorem{theorem}{Theorem}[section]

\newtheorem{lemma}[theorem]{Lemma}
\newtheorem{proposition}[theorem]{Proposition}
\newtheorem{corollary}[theorem]{Corollary}

\theoremstyle{definition}
\newtheorem{definition}[theorem]{Definition}

\newtheorem{remark}[theorem]{Remark}

\DeclareMathOperator{\Ext}{Ext}
\DeclareMathOperator{\Hom}{Hom}

\DeclareMathOperator{\cok}{cok}

\newcommand{\tensor}{\otimes}

\newcommand{\class}[1]{\mathcal{#1}}   

\newcommand{\Z}{\mathbb{Z}}
\newcommand{\Q}{\mathbb{Q/Z}}

\newcommand{\ch}{\textnormal{Ch}(R)}
\newcommand{\chop}{\textnormal{Ch}(R^\circ)}
\newcommand{\cha}[1]{\textnormal{Ch}(\mathcal{#1})}

\newcommand{\dwclass}[1]{dw\widetilde{\class{#1}}}
\newcommand{\exclass}[1]{ex\widetilde{\class{#1}}}

\newcommand{\rightperp}[1]{#1^{\perp}}
\newcommand{\leftperp}[1]{{}^\perp #1}

\newcommand{\homcomplex}{\mathit{Hom}}

\begin{document}

\title{K-flat complexes and derived categories}

\author{James Gillespie}
\address{J.G. \ Ramapo College of New Jersey \\
         School of Theoretical and Applied Science \\
         505 Ramapo Valley Road \\
         Mahwah, NJ 07430\\ U.S.A.}
\email[Jim Gillespie]{jgillesp@ramapo.edu}
\urladdr{http://pages.ramapo.edu/~jgillesp/}

\date{\today}



\begin{abstract}
Let $R$ be a ring with identity. Inspired by recent work in~\cite{emmanouil-K-flatness-and-orthogonality-in-homotopy-cats}, we show that the derived category of $R$ is equivalent to the chain homotopy category of all K-flat complexes with pure-injective components. This is implicitly related to a recollement we exhibit. It expresses $\class{D}_{pur}(R)$, the pure derived category of $R$, as an attachment of  the usual derived category $\class{D}(R)$ with Emmanouil's quotient category $\class{D}_{\textnormal{K-flat}}(R):=K(R)/K\textnormal{-Flat}$, which here we call the K-flat derived category. It follows that this Verdier quotient is a compactly generated triangulated category. We obtain our results by using methods of cotorsion pairs to construct (cofibrantly generated) monoidal abelian model structures on the exact category of chain complexes along with the degreewise pure exact structure. In fact, most of our model structures are obtained as corollaries of a general method which associates an abelian model structure to any class of so-called $\class{C}$-acyclic complexes, where $\class{C}$ is any given class of chain complexes. Finally, we also give a new characterization of K-flat complexes in terms of the pure derived category of $R$.
\end{abstract}

\maketitle

\section{Introduction}\label{sec-intro}

Let $R$ be a ring with identity. In the influential paper~\cite{spaltenstein}, Spaltenstein introduced the K-projective, K-injective, and K-flat chain complexes, providing resolutions for unbounded chain complexes. A K-injective chain complex of $R$-modules can be defined as a chain complex $I$ for which the total Hom complex $\homcomplex(E,I)$ is acyclic (exact) whenever $E$ is acyclic. Given any chain complex of $R$-modules, there exists a homology isomorphism $X \xrightarrow{}I$ where $I$ is a K-injective complex. Dually, there is a homology isomorphism $P \xrightarrow{} X$ where $P$ is a K-projective complex. These resolutions allow one to define derived functors for unbounded complexes.  

The existence of K-injective and K-projective resolutions can be expressed in terms of Bousfield localization (\cite{krause-localization theory for triangulated categories, neeman-book}) in the triangulated category $K(R)$, the homotopy category of chain complexes of $R$-modules. There is a $\Hom_{K(R)}$ orthogonal pair $(K_{ac}(R), K\textnormal{-Inj})$, a so-called Bousfield localizing pair, where $K_{ac}(R)$ denotes the full subcategory of $K(R)$ consisting of all acyclic complexes. It means in particular that $\class{D}(R)$, the derived category of $R$, is equivalent to $K\textnormal{-Inj}$, the full subcategory of $K(R)$ consisting of all K-injectives. On the other hand, we also have a Bousfield localizing pair $(K\textnormal{-Proj}, K_{ac}(R))$. All together it leads to a recollement
\[
\xy
(-28,0)*+{K_{ac}(R)};
(0,0)*+{K(R)};
(25,0)*+{\class{D}(R).};
{(-19,0) \ar (-10,0)};
{(-10,0) \ar@<0.5em> (-19,0)};
{(-10,0) \ar@<-0.5em> (-19,0)};
{(10,0) \ar (19,0)};
{(19,0) \ar@<0.5em> (10,0)};
{(19,0) \ar@<-0.5em> (10,0)};
\endxy
\] 

But $K\textnormal{-Flat}$, the class of all K-flat complexes, has not received as much attention. Recently, in~\cite[Theorem~3.4]{emmanouil-K-flatness-and-orthogonality-in-homotopy-cats}, Emmanouil proves the pleasing fact that $(K\textnormal{-Flat},K_{ac}(\class{PI}))$ is a Bousfield localizing pair in $K(R)$, where $K_{ac}(\class{PI})$ is the full subcategory of $K(R)$ consisting of all acyclic complexes that are homotopy equivalent to a complex of pure-injective $R$-modules. 
It means that the K-flat complexes, which are defined as those complexes $X$ for which the chain complex tensor product $- \otimes_R X$ preserves acyclicity, can be defined in terms of $\Hom_{K(R)}$-orthogonality: They are the left $\Hom_{K(R)}$ orthogonal of $K_{ac}(\class{PI})$. Moreover, it means that the Verdier quotient of $K(R)$ by the thick subcategory of all K-flat complexes is equivalent to $K_{ac}(\class{PI})$. 

In this paper, we also study Emmanouil's quotient category $$\class{D}_{\textnormal{K-flat}}(R):=K(R)/K\textnormal{-Flat},$$ which for convenience we call the \emph{K-flat derived category} of $R$. First, we lift the orthogonality in $K(R)$ between $K\textnormal{-Flat}$ and $K_{ac}(\class{PI})$ to the full chain complex category, $\ch$, by constructing a complete cotorsion pair in this category. More accurately, this cotorsion pair is with respect to  the exact category of chain complexes, along with the short exact sequences that are pure exact in each degree. In fact, this cotorsion pair represents an abelian model structure, in the sense of~\cite{hovey}, and its homotopy category recovers $\class{D}_{\textnormal{K-flat}}(R)$. See Theorem~\ref{them-emman-K-flat}. 
We also build a new monoidal abelian model structure for the usual derived category $\class{D}(R)$, see Theorem~\ref{theorem-derived-cat}. Its cofibrant objects are the K-flat complexes, and its fibrant objects are the chain complexes with a pure-injective $R$-module in each degree. The trivial objects in this model structure are the acyclic complexes, so its homotopy category is the usual derived category $\class{D}(R)$. It follows that $\class{D}(R)$ is equivalent to the chain homotopy category of all K-flat complexes with pure-injective components.

In fact, the model structure for $\class{D}(R)$ is the right Bousfield localization (in the model category sense) of a model structure constructed by Stovicek for $\class{D}_{pur}(R)$, the \emph{pure} derived category of $R$, by our model for $\class{D}_{\textnormal{K-flat}}(R)$. More than that, we show in Theorem~\ref{theorem-recollement} that we have a recollement
\[
\xy
(-28,0)*+{\class{D}_{\textnormal{K-flat}}(R)};
(0,0)*+{\class{D}_{pur}(R)};
(25,0)*+{\class{D}(R)};
{(-19,0) \ar (-10,0)};
{(-10,0) \ar@<0.5em> (-19,0)};
{(-10,0) \ar@<-0.5em> (-19,0)};
{(10,0) \ar (19,0)};
{(19,0) \ar@<0.5em> (10,0)};
{(19,0) \ar@<-0.5em> (10,0)};
\endxy
\] which when restricting the first two categories to fibrant objects becomes 
\[
\xy
(-28,0)*+{K_{ac}(\class{PI})};
(0,0)*+{K(\class{PI})};
(25,0)*+{\class{D}(R)};
{(-19,0) \ar (-10,0)};
{(-10,0) \ar@<0.5em> (-19,0)};
{(-10,0) \ar@<-0.5em> (-19,0)};
{(10,0) \ar (19,0)};
{(19,0) \ar@<0.5em> (10,0)};
{(19,0) \ar@<-0.5em> (10,0)};
\endxy
.\] 
From this we conclude that the triangulated category $\class{D}_{\textnormal{K-flat}}(R)$ is compactly generated. 

In Corollary~\ref{cor-K-flat-char} we prove what the author believes is a new characterization of K-flat chain complexes. They are precisely those complexes that are isomorphic in $\class{D}_{pur}(R)$ to a K-projective complex, or even one with projective components (a so-called DG-projective complex).

Our techniques utilize a combination of categorical purity in the category of chain complexes, along with a use of the degreewise pure exact structure on $\ch$. Theorems~\ref{them-C-acyclic-cot-pair} and~\ref{them-C-acyclic-models} may be of interest in their own right. To describe them, suppose we are given \emph{any} class $\class{C}$ of chain complexes, and let ${}_{\class{C}}\class{W}$ be the class of all chain complexes $W$ for which the chain complex tensor product $C \otimes_R W$ is acyclic for all $C \in \class{C}$. The two theorems show that ${}_{\class{C}}\class{W}$ is the left side of a cotorsion pair that is so nice that it represents a model structure on $\ch$ with ${}_{\class{C}}\class{W}$ the class of trivial objects. To prove these theorems, the author builds upon some techniques he learned from his coauthors in~\cite{estrada-gillespie-odabasi}. 

\section{Preliminaries}\label{sec-preliminaries}
Throughout the paper, $R$ denotes a ring with identity, $R$-Mod the category of left $R$-modules, and $\ch$ the category of chain complexes of left $R$-modules. We denote the opposite ring of $R$ by $R^\circ$, and so we sometimes will write $R^\circ$-Mod to denote the category of right $R$-modules, and simlar for $\chop$.
Our convention is that the differentials of our chain complexes lower degree, so $\cdots
\xrightarrow{} X_{n+1} \xrightarrow{d_{n+1}} X_{n} \xrightarrow{d_n}
X_{n-1} \xrightarrow{} \cdots$ is a chain complex.
Given a chain complex $X$, the
\emph{$n^{\text{th}}$ suspension of $X$}, denoted $\Sigma^n X$, is the complex given by
$(\Sigma^n X)_{k} = X_{k-n}$ and $(d_{\Sigma^n X})_{k} = (-1)^nd_{k-n}$.
For a given $R$-module $M$, we denote the \emph{$n$-disk on $M$} by $D^n(M)$. This is the complex consisting only of $M \xrightarrow{1_M} M$ concentrated in degrees $n$ and $n-1$, and 0 elsewhere. We denote the \emph{$n$-sphere on $M$} by $S^n(M)$, and this is the complex consisting only of $M$ in degree $n$ and 0 elsewhere. 

The chain homotopy category of $R$ is denoted $K(R)$. Recall that its objects are also chain complexes but its morphisms are chain homotopy classes of chain maps.

Given two chain complexes $X, Y \in \ch$, the total  $\Hom$ chain complex will be denoted by $\homcomplex(X,Y)$. It is the complex of abelian groups $$\cdots \xrightarrow{} \prod_{k \in
\Z} \Hom_R(X_{k},Y_{k+n}) \xrightarrow{\delta_{n}} \prod_{k \in \Z}
\Hom_R(X_{k},Y_{k+n-1}) \xrightarrow{} \cdots,$$ where $(\delta_{n}f)_{k}
= d_{k+n}f_{k} - (-1)^n f_{k-1}d_{k}$.
Its homology satisfies $H_n[Hom(X,Y)] = K(R)(X,\Sigma^{-n} Y)$.

We recall too the usual tensor product of chain complexes. Given $X \in \chop$ and $Y \in \ch$, their tensor product $X
\otimes_R Y$ is defined by $$(X \otimes_R Y)_n = \bigoplus_{i+j=n} (X_i
\otimes_R Y_j)$$ in degree $n$. The boundary map $\delta_n$ is defined
on the generators by the formula $\delta_n (x \otimes y) = dx \otimes y +
(-1)^{|x|} x \otimes dy$, where $|x|$ is the degree of the element
$x$.
  
 \subsection{The modified Hom and Tensor complexes}\label{subsec-modified hom and tensor}  The above $\homcomplex$ is often referred to as the \emph{internal hom}, for in the case that $R$ is commutative, $\homcomplex(X,Y)$ is again an object of $\ch$. Note that the cycles in degree 0 of the internal hom coincide with the \emph{external hom} functor: $Z_0[\homcomplex(X,Y)] \cong \Hom_{\ch}(X,Y)$. This idea in fact is useful to define an alternate internal hom as follows. Given $X, Y \in \ch$, we define $\overline{\homcomplex}(X,Y)$ to be the complex $$\overline{\homcomplex}(X,Y)_n = Z_n\homcomplex(X,Y)$$ with differential $$\lambda_n : \overline{\homcomplex}(X,Y)_n \xrightarrow{} \overline{\homcomplex}(X,Y)_{n-1}$$ defined by $(\lambda f)_k = (-1)^nd_{k+n}f_k$. Notice that the degree $n$ component of $\overline{\homcomplex}(X,Y)$ is exactly $\Hom_{\ch}(X,\Sigma^{-n}Y)$. In this way we get an internal hom $\overline{\homcomplex}$ which is useful for categorical considerations in $\ch$. For example, $\overline{\homcomplex}(X,-)$ is a left exact functor, and is exact if and only if $X$ is projective in the category $\ch$. On the other hand, $\overline{\homcomplex}(-,Y)$ is exact if and only if $Y$ is injective in $\ch$. 

Similarly, the usual tensor product of chain complexes does not characterize categorical flatness. For this we need the modified tensor product and its left derived torsion functor from~\cite{enochs-garcia-rozas} and~\cite{garcia-rozas}. We will denote it by $\overline{\otimes}$, and it is defined in terms of the usual tensor product $\otimes_R$ as follows. Given a complex $X$ of right $R$-modules and a complex $Y$ of left $R$-modules, we define $X \overline{\otimes} Y$ to be the complex whose $n^{\text{th}}$ entry is $(X \otimes_R Y)_n / B_n(X \otimes_R Y)$ with boundary map  $(X \otimes_R Y)_n / B_n(X \otimes_R Y) \rightarrow (X \otimes_R Y)_{n-1} / B_{n-1}(X \otimes_R Y)$ given by
\[
\overline{x \otimes y} \mapsto \overline{dx \otimes y}.
\]
This defines a complex and we get a bifunctor $ - \overline{\otimes} - $ which is right exact in each variable. 

\subsection{Cotorsion pairs and abelian model structures}\label{subsection-cot-model}
Besides chain complexes, this paper heavily uses standard facts about \emph{cotorsion pairs} and \emph{abelian model categories}. Standard references for cotorsion pairs include~\cite{enochs-jenda-book} and~\cite{trlifaj-book} and the connection to abelian model categories can be found in~\cite{hovey},~\cite{gillespie-exact model structures}, and~\cite{gillespie-hereditary-abelian-models}. Basic language associated to \emph{exact categories}, in the sense of Quillen, from \cite{quillen-algebraic K-theory}, \cite{keller-exact-cats}, and~\cite{buhler-exact categories} will also be used. It may help the reader to know that, because of~\cite[Appendix~B]{gillespie-G-derived}, an exact structure (in the sense of Quillen) on an abelian category is the same thing as a proper class of short exact sequences (in the sense of Mac\,Lane). So on abelian categories, such as $\ch$, there is no difference between an exact model structure in the sense of~\cite{gillespie-exact model structures} and an abelian model structure compatible with a proper class of short exact sequences as in~\cite{hovey}. 


\section{pure exact structures}

Let $R$ be a ring. We will be considering the category $R$-Mod of (left) $R$-modules along with the class $\class{P}ur$ of pure exact sequences. This gives us an exact structure $(R\textnormal{-Mod}, \class{P}ur)$ which we will denote by $R\textnormal{-Mod}_{pur}$. We let $\class{A}$ denote the class of all $R$-modules and $\class{PI}$ denote the class of all pure-injective $R$-modules. It is immediate that we have a cotorsion pair $(\class{A},\class{PI})$ in $R\textnormal{-Mod}_{pur}$. It is known to be a complete cotorsion pair: For example, for any $R$-module $M$, there is a pure embedding $M \xrightarrow{} M^{++} $ where $M^{++} = \Hom_{\Z}(\Hom_{\Z}(M,\Q), \Q)$ is a pure-injective $R$-module~\cite[Prop.~5.3.9]{enochs-jenda-book}.

The exact structure $R\textnormal{-Mod}_{pur}$ lifts to an exact structure on $\ch$, the category of chain complexes. We let $\ch_{pur}$ denote this exact category whose short exact sequences are pure exact sequences of $R$-modules in each degree. We will denote the associated Yoneda Ext group of all (equivalence classes of) degreewise pure short exact sequences by $\Ext^1_{pur}(X,Y)$. The following proposition tells us that we can use $\Ext^1_{pur}$ to construct complete cotorsion pairs in $\ch_{pur}$ by the usual method of cogenerating by a set.

\begin{proposition}\label{prop-cogeneration-by-set}
Let $\class{S}$ be any set (not a proper class) of chain complexes. Then $(\leftperp{(\rightperp{\class{S}})}, \rightperp{\class{S}})$ is a functorially complete cotorsion pair with respect to the degreewise pure exact structure $\ch_{pur}$. Moreover, the class $\leftperp{(\rightperp{\class{S}})}$ consists precisely of direct summands (retracts) of transfinite degreewise pure extensions of complexes in $\class{S}$. 
\end{proposition}

\begin{proof}
This is a special instance of~\cite[Prop.~5.4]{gillespie-G-derived} and the ``Remark 2'' that follows it. (The functors denoted by $\textnormal{G-}\Ext^1_{\cha{G}}$ in the first paragraph of that proof should really be denoted 
$\textnormal{G-}\Ext^1_{\class{G}}$, as in the paragraph before~\cite[Lemma~4.2]{gillespie-G-derived}.) The category $\cha{G}_G$ of~\cite[Remark 2]{gillespie-G-derived} is discussed in~\cite[Section~4.1]{gillespie-G-derived} and yields $\ch_{pur}$ when the set of generators $\{G_i\}$ is a representative set of all finitely presented $R$-modules.  
\end{proof}

\subsection{$\class{C}$-acyclic complexes}

We will prove a very general theorem and apply it to partially recover a result of Stovicek, see Theorem~\ref{them-stovicek-pure}, and also to obtain a cotorsion pair version of Emmanouil's result that the K-flats are the left side of a Bousfield localizing pair. To do so, we make the following definition. 

\begin{definition}\label{def-C-acyclic}
Let $\class{C}$ be a fixed class of chain complexes of (right) $R$-modules. We will say that a chain complex $W$ of (left) $R$-modules is \emph{$\class{C}$-acyclic} if $C \tensor_R W$ is acyclic for all $C \in \class{C}$. We let ${}_\class{C}\class{W}$ denote the class of all $\class{C}$-acyclic complexes $W$.
\end{definition}

For examples, when $\class{C}$ is the class of all complexes, then the $\class{C}$-acyclic complexes are the pure acyclic complexes; see Definition~\ref{def-k-flat}. If $\class{C}$ is the class of all acyclic complexes, then the $\class{C}$-acyclic complexes are Spaltenstein's K-flat complexes from~\cite{spaltenstein} . And, the $\{S^0(R)\}$-acyclic complexes are the usual acyclic complexes. 
These three examples will play the central role in the next section.

\begin{proposition}\label{prop-C-acyclic-properties}
Let $\class{C}$ be a fixed class of chain complexes of (right) $R$-modules. Then the class ${}_\class{C}\class{W}$ of all $\class{C}$-acyclic complexes satisfies the following properties.
\begin{enumerate}
\item ${}_\class{C}\class{W}$ is closed under direct sums, direct summands (retracts), and direct limits. Also, ${}_\class{C}\class{W}$ contains all contractible complexes. 
\item ${}_\class{C}\class{W}$ is thick in $\ch_{pur}$. That is, it satisfies the 2 out of 3 property with respect to short exact sequences of chain complexes that are pure exact in each degree.
\item ${}_\class{C}\class{W}$ is closed under transfinite extensions in $\ch_{pur}$. That is, if $X$ can be written as a transfinite extension $X \cong \varinjlim_{\alpha < \lambda} X_{\alpha}$ where each $X_{\alpha} \xrightarrow{} X_{\alpha +1}$ is a degreewise pure monomorphism, and $X_0$ and each $X_{\alpha +1}/X_{\alpha}$ are $\class{C}$-acyclic, then $X$ too is $\class{C}$-acyclic. 
\end{enumerate}
\end{proposition} 
 
\begin{proof}
The class of exact complexes of abelian groups is closed under direct sums, direct summands, and direct limits. For any chain complex $C$, the functor $C \otimes_R -$ commutes with direct sums, direct summands, and direct limits. So ${}_\class{C}\class{W}$ is closed under these operations. As for the contractible complexes being $\class{C}$-acyclic, recall that any contractible complex $X$ must take the form $X \cong \bigoplus_{n\in\Z}D^n(M_n)$ for some modules $\{M_n\}_{n\in\Z}$. So then we have  $C \otimes_R X \cong \bigoplus_{n\in\Z} (C\otimes_R D^n(M_n))$, and each $C\otimes_R D^n(M_n)$ can be shown to be acyclic. For example, see~\cite[Exercise~1.2.5]{weibel}.

Now let $C \in \class{C}$, and consider a short exact sequence of chain complexes 
$$0 \xrightarrow{} W \xrightarrow{} X \xrightarrow{} Y \xrightarrow{} 0,$$ pure exact in each degree. Then for all pairs of integers $i, j$ we have short exact sequences
$$0 \xrightarrow{} C_i \tensor_R W_j \xrightarrow{}C_i \tensor_R X_j \xrightarrow{} C_i \tensor_R Y_j \xrightarrow{} 0.$$
Since short exact sequences are closed under direct sums, it follows that for each $n \in \Z$ we have short exact sequences
$$0 \xrightarrow{} \bigoplus_{i+j=n}  C_i \tensor_R W_j \xrightarrow{} \bigoplus_{i+j=n} C_i \tensor_R X_j \xrightarrow{} \bigoplus_{i+j=n} C_i \tensor_R Y_j \xrightarrow{} 0.$$
By definition, this is the degree $n$ component of the tensor product, so we get a short exact sequence of chain complexes
$$0 \xrightarrow{} C \tensor_R W \xrightarrow{} C \tensor_R X \xrightarrow{} C \tensor_R Y \xrightarrow{} 0.$$
So then if 2 out of 3 of these are acyclic, so is the third. This shows ${}_\class{C}\class{W}$ satisfies the 2 out of 3 property on short exact sequences. 

So now it is easy to see that ${}_\class{C}\class{W}$ is closed under transfinite extensions in $\ch_{pur}$. This follows from the fact that we have shown ${}_\class{C}\class{W}$ to be closed under degreewise pure extensions, and under direct limits. 
\end{proof} 
 
\subsection{Categorical purity}

We've been considering the exact category $\ch_{pur}$. It is the category of chain complexes along with the proper class of all degreewise pure exact sequences of $R$-modules. But there is also the (stronger) categorical notion of purity in the full abelian category $\ch$, of chain complexes or $R$-modules. We are referring to the categorical tensor product $X\overline{\otimes}\, Y$ introduced in~\cite{enochs-garcia-rozas}, studied in~\cite{garcia-rozas}, and used in~\cite[Section~4]{gillespie}. See Section~\ref{subsec-modified hom and tensor} for its definition.

\begin{definition}\label{def-purity}
A short exact sequence $0 \xrightarrow{} P \xrightarrow{} X \xrightarrow{} Q \xrightarrow{} 0$ is called \emph{categorically pure} if for any chain complex $Y$ of (right) $R$-modules, the sequence 
$$0 \xrightarrow{} Y\overline{\otimes}\,P \xrightarrow{} Y\overline{\otimes}\,X \xrightarrow{} Y\overline{\otimes}\,Q \xrightarrow{} 0$$ remains exact. 
\end{definition}
 
This is indeed the categorical purity in  $\ch$ as such short exact sequences are precisely the direct limits of split exact sequences ending in a finitely presented chain complex. This and other basic properties of categorical purity are listed in~\cite[Prop.~4.4]{gillespie}. 

\begin{lemma}\label{lemma-categorical-degreewise-purity}
Any categorical  pure exact sequence is a degreewise pure exact sequence. That is, given any pure exact sequence as in Definition~\ref{def-purity}, the short exact sequence
$$0 \xrightarrow{} P_n \xrightarrow{} X_n \xrightarrow{} Q_n \xrightarrow{} 0$$
 is pure exact in $R$-Mod for each $n$.
\end{lemma}

\begin{proof}
Let $F$ be a finitely presented $R$-module. It is enough to show that the functor $\Hom_R(F,-)$ preserves the displayed short exact sequence for each $n$. By a standard adjunction, see~\cite[Lemma~3.1(1)]{gillespie}, it is equivalent to show that $\Hom_{\ch}(D^n(F),-)$ preserves the former sequence from Definition~\ref{def-purity}. But this follows from~\cite[Prop~4.4(3)]{gillespie}; categorically pure exact seqeuences are characterized by the condition that $\Hom_{\ch}(L,-)$ preserves their exactness whenever $L$ is a finitely presented chain complex. 
\end{proof}
 
\begin{proposition}\label{prop-categorical-pure-subcomplexes-quotients}
The class ${}_\class{C}\class{W}$ of all $\class{C}$-acyclic complexes is closed under categorically pure subcomplexes and quotients. That is, 
for any categorically pure short exact sequence  $0 \xrightarrow{} P \xrightarrow{} W \xrightarrow{} Q \xrightarrow{} 0$, the complexes $P$ and $Q$ must both be $\class{C}$-acyclic whenever $W$ is $\class{C}$-acyclic. 
\end{proposition}

\begin{proof}
By Lemma~\ref{lemma-categorical-degreewise-purity}, the given sequence is an admissible short exact sequence in $\ch_{pur}$, that is, it is degreewise pure. So given \emph{any} complex $C$, it follows, just like in the proof of Proposition~\ref{prop-C-acyclic-properties}, that we get a short exact sequence for each $n$:
\begin{equation}\label{equation-ses1}\tag{$\dagger$} 0 \xrightarrow{} (C \tensor_R P)_n \xrightarrow{} (C \tensor_R W)_n \xrightarrow{} (C \tensor_R Q)_n \xrightarrow{} 0. \end{equation}  
But also, for \emph{any} complex $C$, we get a short exact sequence $$0 \xrightarrow{} C\overline{\otimes}\,P \xrightarrow{} C\overline{\otimes}\,W \xrightarrow{} C\overline{\otimes}\,Q \xrightarrow{} 0.$$
By the very definition of $\overline{\otimes}$, this means we have another short exact sequence for each $n$:
\begin{equation}\label{equation-ses2}\tag{$\dagger\dagger$} 0 \xrightarrow{} \frac{(C \tensor_R P)_n}{B_n(C \tensor_R P)} \xrightarrow{} \frac{(C \tensor_R W)_n}{B_n(C \tensor_R W)} \xrightarrow{} \frac{(C \tensor_R Q)_n}{B_n(C \tensor_R Q)} \xrightarrow{} 0. \end{equation}
There is an epimorphism from the short exact sequence \eqref{equation-ses1}, to the short exact sequence \eqref{equation-ses2}.
So by the snake lemma we get another short exact sequence for each $n$:
\begin{equation}\label{equation-ses3}\tag{$*$} 0 \xrightarrow{} B_n(C \tensor_R P) \xrightarrow{} B_n(C \tensor_R W) \xrightarrow{} B_n(C \tensor_R Q) \xrightarrow{} 0. \end{equation}  
Now in the same way, there is an epimorphism from the short exact sequence \eqref{equation-ses1}, to the short exact sequence \eqref{equation-ses3} in degree $n-1$, giving us another short exact sequence for each $n$:
\begin{equation}\label{equation-ses4}\tag{$**$} 0 \xrightarrow{} Z_n(C \tensor_R P) \xrightarrow{} Z_n(C \tensor_R W) \xrightarrow{} Z_n(C \tensor_R Q) \xrightarrow{} 0. \end{equation} 
Finally, we apply the snake lemma one more time: There is a monomorphism from the short exact sequence \eqref{equation-ses3} to the short exact sequence \eqref{equation-ses4}. Its cokernel is the short exact sequence of homology:
$$0 \xrightarrow{} H_n(C \tensor_R P) \xrightarrow{} H_n(C \tensor_R W) \xrightarrow{} H_n(C \tensor_R Q) \xrightarrow{} 0.$$ 
In the case that $C \in \class{C}$ and $W$ is $\class{C}$-acylic, it means $H_n(C \tensor_R W) = 0$. So the above short exact sequence implies $H_n(C \tensor_R P) = H_n(C \tensor_R Q) = 0$ too. In other words, $P$ and $Q$ are each $\class{C}$-acyclic whenever $W$ is $\class{C}$-acyclic. 
\end{proof}
 
\subsection{$\class{C}$-acyclic cotorsion pairs} For the following theorems we continue to let $\class{C}$ denote a fixed class of chain complexes of (right) $R$-modules and ${}_\class{C}\class{W}$ the class of all $\class{C}$-acyclic complexes. So again, $W \in {}_\class{C}\class{W}$ means that the chain complex tensor product $C \otimes_R W$ is acyclic for all $C \in \class{C}$. We set $\class{F} :=\rightperp{{}_\class{C}\class{W}}$, in $\ch_{pur}$. That is, 
 $$\class{F} = \{\,F \in \ch\,|\, \Ext^1_{pur}(W,F)=0\  \forall \,W \in \,{}_\class{C}\class{W}\},$$ and we recall that $\Ext^1_{pur}(W,F)$ is the group of degreewise pure  (not necessarily categorically pure) short exact sequences.
 
 \begin{theorem}\label{them-C-acyclic-cot-pair} 
$({}_\class{C}\class{W}, \class{F})$ is a complete hereditary cotorsion pair, cogenerated by a set, in the exact category $\ch_{pur}$. 
 \end{theorem}

 \begin{proof}
For a chain complex $X$,  we define its cardinality to be $|X| :=|\coprod_{n\in\Z} X_n|$.
Let $\kappa \geq |R|$ be an infinite cardinal. Up to isomorphism, we can find a set $\class{S}$ (as opposed to a proper class) of $\class{C}$-acyclic chain complexes $X$ with cardinality $|X| \leq \kappa$. By Proposition~\ref{prop-cogeneration-by-set}, it cogenerates a functorially complete cotorsion pair $(\leftperp{(\rightperp{\class{S}})}, \rightperp{\class{S}})$ in the exact category $\ch_{pur}$. Moreover, the class $\leftperp{(\rightperp{\class{S}})}$ consists precisely of direct summands (retracts) of transfinite degreewise pure extensions of complexes in $\class{S}$. We will show that $(\leftperp{(\rightperp{\class{S}})}, \rightperp{\class{S}}) = ({}_\class{C}\class{W}, \class{F})$.

Of course, it is enough to show $\leftperp{(\rightperp{\class{S}})} = {}_\class{C}\class{W}$. The containment $\leftperp{(\rightperp{\class{S}})} \subseteq {}_\class{C}\class{W}$ follows from Proposition~\ref{prop-C-acyclic-properties}, because $\class{S} \subseteq {}_\class{C}\class{W}$ and the class ${}_\class{C}\class{W}$ is closed under direct summands and transfinite degreewise pure extensions.  

On the other hand, we will show ${}_\class{C}\class{W} \subseteq \leftperp{(\rightperp{\class{S}})}$ by showing that every $X \in {}_\class{C}\class{W}$ is a transfinite degreewise pure extension of complexes in $\class{S}$.
But with standard tools and techniques at our disposal, it is easier to show the stronger statement: each $X \in {}_\class{C}\class{W}$ is a transfinite \emph{categorically pure} extension of complexes in $\class{S}$. Moreover, by Lemma~\ref{lemma-categorical-degreewise-purity}, this will prove the weaker statement we want. 

So given $X\in
{}_\class{C}\class{W}$, we use a standard technique to define a strictly increasing chain $X_{i}\subseteq X$
with $X_{i}\in {}_\class{C}\class{W}$, by transfinite induction on $i$.  For $i=0$,
we let $X_{0}$ be a nonzero categorically pure subcomplex of $X$ with $|X_0| \leq \kappa$ (using~\cite[Lemma~4.6]{gillespie}, so~\cite[Lemma~5.2.1]{garcia-rozas}).  Being a categorically pure subcomplex of the $\class{C}$-acyclic $X$, we note that $X_0$ must also be $\class{C}$-acyclic, by Proposition~\ref{prop-categorical-pure-subcomplexes-quotients}.
Having defined the categorically pure subcomplex $X_{i}$ of $X$ and assuming that $X_{i}\neq
X$, we let $N_{i}$ be a nonzero categorically pure subcomplex of $X/X_{i}$ with $|N_i| \leq \kappa$. Since $X_{i}$ is a categorically pure subcomplex of
$X$, Proposition~\ref{prop-categorical-pure-subcomplexes-quotients} assures us that $X/X_{i}$ is also in ${}_\class{C}\class{W}$, and that $N_{i}$ is as well.  We then
let $X_{i+1}$ be the preimage in $X$ of $N_{i}$, so that $X_{i+1}$ is
a categorically pure subcomplex of $X$~\cite[Prop.6.6(3)]{gillespie-G-derived}, so also in ${}_\class{C}\class{W}$.  For the limit ordinal
step, we define $X_{\beta}= \bigcup_{\alpha <\beta}X_{\alpha}$; this
is a colimit of categorically pure subcomplexes of $X$ so is also a categorically pure subcomplex of
$X$.  This process will eventually stop when $X_{i}=X$, at which point
we have written $X$ as a transfinite extension of complexes in $\class{S}$. 

This proves $({}_\class{C}\class{W}, \class{F})$ is a complete cotorsion pair in $\ch_{pur}$. It is certainly hereditary since we even showed in Proosition~\ref{prop-C-acyclic-properties} that ${}_\class{C}\class{W}$ is thick in $\ch_{pur}$.
\end{proof}

\begin{remark}
Assume $\class{C}$ is closed under suspensions and set $$\class{C}^+ = \{\,\Hom_{\Z}(C,\Q) \,|\, C \in \class{C}\,\}.$$
Then the cotorsion pair $({}_\class{C}\class{W}, \class{F})$ of Theorem~\ref{them-C-acyclic-cot-pair} is \emph{generated} by $\class{C}^+$ in  the sense that ${}_\class{C}\class{W} = \leftperp{(\class{C}^+)}$, in $\ch_{pur}$. (Reason:) For $C \in \class{C}$, the Pontryagin dual $\Hom_R(C,\Q)$  is a complex of pure-injectives. So an argument similar to the one in~\cite[Prop.~3.2]{emmanouil-K-flatness-and-orthogonality-in-homotopy-cats} invoking adjoint associativity will work.   
\end{remark}

 \begin{theorem}\label{them-C-acyclic-models}
In fact, $({}_\class{C}\class{W}, \class{F})$ is an injective cotorsion, meaning that 
$$\mathfrak{M} = (All, {}_\class{C}\class{W}, \class{F})$$ is an abelian (equivalently, exact) model structure on the exact category $\ch_{pur}$.
This is a cofibrantly generated model structure. Its homotopy category is a well-generated triangulated category, and equivalent to the full subcategory of $K(R)$ consisting of all complexes (homotopy equivalent to one) in $\class{F}$.
\end{theorem}

\begin{proof}
As pointed out in Section~\ref{subsection-cot-model}, for abelian categories we need not distinguish between an exact model structure in the sense of~\cite{gillespie-exact model structures} and an abelian model structure compatible with a proper class of short exact sequences as in~\cite{hovey}. So we use the usual Hovey's correspondence between cotorsion pairs and abelian model structures. 

We already showed ${}_\class{C}\class{W}$ is thick. By\cite[Lemma~5.7(4)]{gillespie-G-derived}, the injective objects in the exact category $\ch_{pur}$ are precisely the contractible complexes with pure-injective components, and they make up the right hand side of a complete cotorsion pair in $\ch_{pur}$. 

 So all that is left is to show that ${}_\class{C}\class{W} \cap \class{F}=\class{I}$, where $\class{I}$ denotes the class of all contractible complexes with pure-injective components. By Proposition~\ref{prop-C-acyclic-properties}, ${}_\class{C}\class{W}$ contains all contractible complexes, so it follows easily that ${}_\class{C}\class{W} \cap \class{F}\supseteq \class{I}$. To prove the reverse containment ${}_\class{C}\class{W} \cap \class{F}\subseteq \class{I}$, let $X \in {}_\class{C}\class{W} \cap \class{F}$ and write a degreewise pure short exact sequence $0 \xrightarrow{} X \xrightarrow{}  I \xrightarrow{}  I/X
\xrightarrow{}  0$ where $I$ is injective in $\ch_{pur}$. That is, $I$ is a contractible complex of pure-injectives, and so as already noted, $I \in {}_\class{C}\class{W}$. But
${}_\class{C}\class{W}$ is a thick class in $\ch_{pur}$, which implies $I/X \in
{}_\class{C}\class{W}$. But now since $({}_\class{C}\class{W},\class{F})$ is a cotorsion pair
the short exact sequence splits. Therefore $X$ is a direct summand of
$I$, and this proves $X \in \class{I}$.

Since the cotorsion pairs are each cogenerated by a set it follows that the model structure is cofibrantly generated~\cite[Section~6]{hovey}. Moreover, since we have a cofibrantly generated model structure on a locally presentable (pointed) category, it follows from~\cite{rosicky-brown representability combinatorial model srucs} that the homotopy category is well generated in the sense of~\cite{neeman-well generated}. 

It is left to prove the claim that $\textnormal{Ho}(\mathfrak{M})$ is equivalent to the full subcategory of $K(R)$ consisting of all complexes (homotopy equivalent to one) in $\class{F}$. From the fundamental theorem of model categories~\cite[Theorem~1.2.10]{hovey-model-categories} we know that the homotopy category of this model structure is equivalent to $\class{F}/\sim$ where $\sim$ denotes the formal homotopy relation. However, it follows from~\cite[Corollary~4.8]{gillespie-exact model structures} that formally $f \sim g$ if and only if $g-f$ factors through an injective object in $\ch_{pur}$. As previously noted, these are the contractible complexes with pure-injective components. But since they are contractible, and since moreover all contractible complexes are trivial (Proposition~\ref{prop-C-acyclic-properties}), this implies that $f \sim g$ if and only if $f$ and $g$ are chain homotopic in the usual sense; see~\cite[Lemma~5.1]{gillespie-hereditary-abelian-models}. So the homotopy category is as described.
\end{proof} 

 As an aside, in the standard language from~\cite{enochs-jenda-book} and~\cite{garcia-rozas}  we can assert that each complex has a $\class{C}$-acyclic cover. This follows from Theorem~\ref{them-C-acyclic-cot-pair} and the fact that ${}_\class{C}\class{W}$ is closed under direct limits (Proposition~\ref{prop-C-acyclic-properties}).
\begin{corollary}
For any class of complexes $\class{C}$, the class ${}_\class{C}\class{W}$ is a covering class. That is, each complex has a $\class{C}$-acyclic cover. 
\end{corollary}

\section{K-flat, pure acyclic, and DG-pure-injective complexes}
With respect to the pure exact structure on $R$-modules, an \emph{acyclic} complex in the formal sense of~\cite[Def.~10.1]{buhler-exact categories}, is precisely a \emph{pure acyclic complex} in the usual sense. Let us record this definition. 

\begin{definition}\label{def-k-flat}
Let $W$ be a chain complex of left  $R$-modules.
\begin{enumerate}
\item $W$ is called \emph{pure acyclic} if $M \tensor_R W$ is acyclic for all $R^\circ$-modules $M$. It is equivalent to require that $X \tensor_R W$ be acyclic for all chain complexes of $R^\circ$-modules $X$. (See Prop.~\ref{proposition-pure-acyclic-complexes} below.) 
We denote the class of all the pure acyclic complexes by $\class{A}_{pur}$. 
\item $W$ is called \emph{K-flat} if $E \tensor_R W$ is acyclic for all acyclic chain complexes of $R^\circ$-modules $E$. 
We denote the class of all K-flat complexes by $\class{KF}$.
\end{enumerate} 
\end{definition}

The two notions are nicely related as follows. 

\begin{proposition}\label{proposition-pure-acyclic-complexes}
The following are equivalent for a chain complex $W$.
\begin{enumerate}
\item $W \in \class{A}_{pur}$. That is, $M \tensor_R W$ is acyclic for all $R^\circ$-modules $M$.
\item $X \tensor_R W$ is acyclic for all chain complexes of $R^\circ$-modules $X$.
\item  $W$ is acyclic and K-flat.
\end{enumerate}
\end{proposition}

\begin{proof}
These statements appear as Exercises~5.4.16 and~5.4.17 in~\cite{christensen-foxby-holm-book}. See~\cite[Prop.~1.1]{emmanouil-relation-K-flatness-K-projectivity} for a clear proof. 
\end{proof}

\subsection{Complexes of pure-injective $R$-modules}
Let $\dwclass{PI}$ denote the class of all degreewise pure-injective chain complexes. That is, $X \in \dwclass{PI}$ if each $X_n \in \class{PI}$. Stovicek proved the following in~\cite{stovicek-purity}.

\begin{theorem}[Stovicek~\cite{stovicek-purity} Theorem~5.4]\label{them-stovicek-pure}
The pair $(\class{A}_{pur},\dwclass{PI})$ is an injective cotorsion in $\ch_{pur}$, meaning that 
$$\mathfrak{M}_{\textnormal{pur}} = (All, \class{A}_{pur},\dwclass{PI})$$ is a (cofibrantly generated) abelian model structure on the exact category $\ch_{pur}$.
Its homotopy category recovers $\class{D}_{pur}(R)$, the pure derived category of $R$, and shows it to be equivalent to $K(\class{PI})$, the full subcategory of $K(R)$ consisting of all complexes (homotopy equivalent to one) in $\dwclass{PI}$. 

Assuming $R$ is commutative, this model structure is in fact monoidal with respect to the usual tensor product of chain complexes. 
\end{theorem}

\begin{proof}
In light of Proposition~\ref{proposition-pure-acyclic-complexes} we get ${}_\class{C}\class{W} = \class{A}_{pur}$, by taking $\class{C}$ to be the class of \emph{all} chain complexes of (right) $R$-modules.  Therefore, Theorem~\ref{them-C-acyclic-models} produces a model structure $\mathfrak{M}_{pur} = (All, \class{A}_{pur}, \rightperp{\class{A}_{pur}})$ on the exact category $\ch_{pur}$, with the stated properties. Stovicek showed that the class of fibrant objects, $\rightperp{\class{A}_{pur}}$, is precisely $\dwclass{PI}$. Briefly, Stovicek's argument can be summarized as follows: Each pure acyclic complex is a direct limit of contractible complexes; see also~\cite[Proposition~2.2]{emmanouil-pure-acyclic-complexes}. So then for a given $Y \in \dwclass{PI}$, the kernel of the functor $\Ext^1_{pur}(-,Y)$ contains all contractible complexes, and one checks that the argument from~\cite[Proposition~3.1]{gillespie-ding-modules} readily adapts to show that the kernel of the functor $\Ext^1_{pur}(-,Y)$ is closed under direct limits. 

Assuming $R$ is commutative, we are asserting that the model structure is monoidal with respect to the usual tensor product of chain complexes. We refer the reader to the proof ahead of the same fact for the model structure in Theorem~\ref{theorem-derived-cat}. The proof is almost identical. 
\end{proof}

\subsection{Acyclic complexes of pure-injective $R$-modules}
We now give a model category interpretation of Emmanouil's result from~\cite{emmanouil-K-flatness-and-orthogonality-in-homotopy-cats}: The class of all K-flat complexes is orthogonal to the class of all acyclic complexes of pure-injectives. 
 
Recall that we are letting $\class{KF}$ denote the class of all K-flat complexes. As pointed out in~\cite{emmanouil-K-flatness-and-orthogonality-in-homotopy-cats}, $\class{KF}$ determines a triangulated subcategory of the chain homotopy category $K(R)$. Thus we can form the Verdier quotient $$\class{D}_{\textnormal{K-flat}}(R) := K(R)/\class{KF},$$ which we will call the \emph{K-flat derived category of $R$}.  Let us say that a chain map $f : X \xrightarrow{} Y$ is a \emph{K-flat quism} if $E \otimes_R X \xrightarrow{E\otimes_R f} E \otimes_R Y$ is a homology isomorphism for all acyclic chain complexes of $R^\circ$-modules $E$. We note that $f$ is a K-flat quism if and only if its cone is K-flat. One can argue this from the standard (degreewise split) short exact sequence containing cone($f$), as in~\cite[Paragraph~1.5.2]{weibel}.

We let $\exclass{PI}$ denote the class of all exact complexes of pure-injective $R$-modules. That is $X \in \exclass{PI}$ if $X$ is acyclic (just exact, not necessarily pure acyclic) and each $X_n \in \class{PI}$.

\begin{theorem}\label{them-emman-K-flat}
The pair $(\class{KF},\exclass{PI})$ is an injective cotorsion in $\ch_{pur}$, meaning that 
$$\mathfrak{M}_{\class{KF}} = (All, \class{KF},\exclass{PI})$$ is a (cofibrantly generated) abelian model structure on the exact category $\ch_{pur}$. The weak equivalences in this model structure are precisely the K-flat quisms.  Its homotopy category recovers $\class{D}_{\textnormal{K-flat}}(R)$, the K-flat derived category of $R$, and shows it to be equivalent to $K_{ac}(\class{PI})$, the full subcategory of $K(R)$ consisting of all acyclic complexes (homotopy equivalent to one) in $\dwclass{PI}$.
\end{theorem}

\begin{proof}
We get ${}_\class{C}\class{W} = \class{KF}$, by taking $\class{C}$ to be the class of all \emph{acyclic} chain complexes of (right) $R$-modules.  Therefore, Theorem~\ref{them-C-acyclic-models} produces a model structure $\mathfrak{M}_{\class{KF}} = (All, \class{KF}, \rightperp{\class{KF}})$ on the exact category $\ch_{pur}$, with the stated properties. Here, $\rightperp{\class{KF}}$ is the right Ext-orthogonal, in $\ch_{pur}$,  of the class of all K-flat complexes. 

It is left to show $\rightperp{\class{KF}} = \exclass{PI}$. Since the class $\class{KF}$ of K-flat complexes contains $\{S^n(R)\}$ it follows that every complex in $\rightperp{\class{KF}}$ is exact (acyclic). Since $\class{KF}$ contains $\class{A}_{pur}$ it follows that $\rightperp{\class{KF}} \subseteq \rightperp{\class{A}_{pur}} = \dwclass{PI}$. This shows $\rightperp{\class{KF}} \subseteq \exclass{PI}$.

To show $\rightperp{\class{KF}} \supseteq \exclass{PI}$, let $X \in \exclass{PI}$. Using the complete cotorsion pair  $(\class{KF}, \rightperp{\class{KF}})$ we can find a (degreewise pure) short exact sequence of complexes
$$0 \xrightarrow{} X \xrightarrow{} Y \xrightarrow{} K \xrightarrow{}0$$
where $K$ is K-flat and $Y \in \rightperp{\class{KF}}$. As already shown above, $Y$ is acyclic. Also $X$ is acyclic by hypothesis, and therefore $K$ is an acyclic K-flat complex. By Proposition~\ref{proposition-pure-acyclic-complexes}, $K$ is pure acyclic. Since $(\class{A}_{pur}, \dwclass{PI})$ is a cotorsion pair, the short exact sequence splits! We conclude $X \in \rightperp{\class{KF}}$ since any Ext-othogonal class is closed under direct summands. 

Let us show that the weak equivalences in $\mathfrak{M}_{\class{KF}}$ are precisely the K-flat quisms.
Given any morphism $f : X \xrightarrow{} Y$, we may use the model structure to factor it as $f = pi$ where $i$ is a cofibration and $p$ is a trivial fibration: $$\begin{tikzcd}
X  \arrow[rd, "i" ', tail]  \arrow[rr, "f"] &   & Y  \\
& Z \arrow[ru, "p" ', two heads] \arrow[rd, two heads] & \\
C \arrow[ru, tail] & & Q \\
\end{tikzcd}$$
This is a commutative diagram with two degreewise pure exact sequences, with $Q := \cok{i}$, and with 
$C := \ker{p} \in \class{KF}\bigcap\exclass{PI}$. This means $C$ must be a contractible complex of pure-injective modules. 
Since the sequences are degreewise pure, for any complex $X$, applying the functor $X \tensor_R -$ yields a similar commutative diagram with the two short exact sequences fixed. Note in particular, that taking the $X = E$ to be acyclic complexes, $E \tensor_R p$ is a homology isomorphism since $E \tensor_R C$ will always be exact ($C$ is contractible). In other words, $p$ is necessarily a K-flat quism. By the 2 out of 3 property, $E \tensor_R f$ will be a homology isomorphism for all acyclic $E$ ($f$ will be a K-flat quism) if and only if $E \tensor_R i$ is a homology isomorphism for all acyclic $E$ ($i$ is a K-flat quism). This happens if and only if $E \tensor_R Q$ is acyclic for all acyclic $E$; in other words, if and only if $Q \in \class{KF}$, the class of trivially cofibrant objects. Summarizing, $f$ will be a K-flat quism if and only if $i$ is a trivial cofibration which, by the 2 out of 3 property on weak equivalences, is equivalent to $f$ being a weak equivalence.

Finally, let us make explicit why the homotopy category $\textnormal{Ho}(\mathfrak{M}_{\class{KF}})$ coincides with the Verdier quotient $\class{D}_{\textnormal{K-flat}}(R) := K(R)/\class{KF}$. For this we explain why $\textnormal{Ho}(\mathfrak{M}_{\class{KF}})$ satisfies the universal property that is unique to the Verdier quotient functor $K(R) \xrightarrow{} K(R)/\class{KF}$. First, note that the canonical functor $\gamma : \ch_{pur} \xrightarrow{} \textnormal{Ho}(\mathfrak{M}_{\class{KF}})$ factors through $K(R)$. Indeed any chain homotopy equivalence is a K-flat quism, and this leads to a unique factorization $\gamma = \bar{\gamma}\circ\pi$:
$$\ch_{pur} \xrightarrow{\pi} K(R) \xrightarrow{\bar{\gamma}} \textnormal{Ho}(\mathfrak{M}_{\class{KF}}).$$
The point is to see that the canonical functor $K(R) \xrightarrow{\bar{\gamma}} \textnormal{Ho}(\mathfrak{M}_{\class{KF}})$ is universal among triangulated functors that ``kill'' $\class{KF}$.
But  the canonical functor $\gamma : \ch_{pur} \xrightarrow{} \textnormal{Ho}(\mathfrak{M}_{\class{KF}})$ already satisfies a similar universal property: It takes short exact sequences in $\ch_{pur}$ to exact triangles in $\textnormal{Ho}(\mathfrak{M}_{\class{KF}})$ and is universal among such functors that ``kill'' $\class{KF}$; see \cite[Prop.~3.2]{gillespie-recoll2}. So now given any triangulated functor $F : K(R) \xrightarrow{} \class{T}$ with $\ker{F} \supseteq \class{KF}$, we note $F\circ\pi$  takes short exact sequences in $\ch_{pur}$ to exact triangles and ``kills'' $\class{KF}$. Hence it induces a unique functor $\bar{F}:\textnormal{Ho}(\mathfrak{M}_{\class{KF}}) \xrightarrow{} \class{T}$ satisfying $\bar{F}\circ\gamma = F\circ\pi$. Now we have $F\circ\pi =  \bar{F}\circ\bar{\gamma}\circ\pi$, and $\pi$ is right cancellable. Therefore $F =  \bar{F}\circ\bar{\gamma}$, proving $K(R) \xrightarrow{\bar{\gamma}} \textnormal{Ho}(\mathfrak{M}_{\class{KF}})$ satisfies the universal propery unique to the Verdier quotient $K(R)/\class{KF}$.
\end{proof}

We have the following interesting consequence of the above theorems. Recall that a DG-projective complex is a K-projective complex with projective components.  
 
\begin{corollary}\label{cor-K-flat-char}
$K\textnormal{-Flat}$ is the isomorphic closure of $K\textnormal{-Proj}$ in $\class{D}_{pur}(R)$. In fact, $X$ is K-flat if and only if there is an (epimorphic) pure homology isomorphism $f : P \xrightarrow{} X$ where $P$ is DG-projective.
\end{corollary}

\begin{proof}
Any chain map $f : P \xrightarrow{} X$ is a weak equivalence in the Theorem~\ref{them-stovicek-pure} model structure if and only if $f$ is a pure homology isomorphism.  Write a factorization $f = pi$ where $i: P \xrightarrow{} Y$ is a degreewise pure monic with $\cok{i}$ pure acyclic (trivial cof.), and $p : Y \xrightarrow{} X$ is a degreewise pure epimorphism with $\ker{p}$ a complex of pure-injectives (a fibration).  By the 2 out of 3 axiom, $f$ is a weak equivalence if and only if $p$ is, and this would force $\ker{p}$ to be a contractible complex with pure-injective components. So in this case, since  the class K-Flat is thick in $\ch_{pur}$, we can easily argue that if $P$ (resp. $X$) is any K-flat complex (which includes the possibility of being DG-projective) then $X$ (resp. $P$) must be K-flat. 

On the other hand, suppose $X$ is given to be K-flat. It is standard that we can find an epimorphic homology isomorphism $f : P \xrightarrow{} X$ with $P$ a DG-projective complex (or just K-projective or DG-flat). In any case, $\ker{f}$ is acyclic. One verifies that in any factorization  $f = pi$ as in the start of the last paragraph, $\ker{p}$ is in fact an extension of $\ker{f}$ and $\cok{i}$. It follows that $\ker{p}$ is an acyclic complex of pure-injectives. But then by Theorem~\ref{them-emman-K-flat}, the epimorphism $p$ must split. So  $\ker{p}$ must also be K-flat, as it is a direct summand of $Y$ (which is K-flat because it is a degreewise pure extension of the K-flat complexes $P$ and $\cok{i}$). Hence $\ker{p}$ must be a contractible complex of pure-injectives. It proves $p$ is a pure homology isomorphism, and so $f$ must be too. 
\end{proof}

The author would like to point out that glimpses of Corollary~\ref{cor-K-flat-char} can already be seen in~\cite{emmanouil-relation-K-flatness-K-projectivity}. In particular, it can be deduced by combining~\cite[Lemma~3.1]{emmanouil-relation-K-flatness-K-projectivity} and~\cite[Remark~3.2(ii)]{emmanouil-relation-K-flatness-K-projectivity}. Another closely related characterization of K-flatness is given in~\cite[Proposition~3.3]{emmanouil-relation-K-flatness-K-projectivity}.

\subsection{DG-pure-injective complexes}
We start with a definition and a lemma. 

\begin{definition}
We will say that a chain complex $J$ is \emph{DG-pure-injective} if each $J_n$ is a pure-injective $R$-module and if any chain map $f : E \xrightarrow{} J$ from an exact complex $E$ is null homotopic. 
\end{definition}

Note then that any DG-injective complex in the usual sense is DG-pure-injective. In fact, recall that a DG-injective complex is the same thing as a K-injective complex (in the sense of Spaltenstein~\cite{spaltenstein}) with injective components. In the same way we have the following. 

\begin{lemma}
A chain complex $J$ is DG-pure-injective if and only if it is K-injective and each component $J_n$ is a pure-injective module.
\end{lemma}
 
 \begin{proof}
 The definition of DG-pure-injective is equivalent to the statement that each $J_n$ is a pure-injective and $\homcomplex(E,J)$ is exact for all exact complexes $E$. This is equivalent to saying each $J_n$ is pure-injective and the morphism set $K(R)(E,J)=0$ for all exact complexes $E$. This means $J$ is K-injective, by definition.  
 \end{proof}
 
 \begin{corollary}\label{cor-DG-pure-injective-model}
 Let $\class{E}$ denote the class of all exact (acyclic) chain complexes, and let $\class{F}$ denote the class of all DG-pure-injective complexes.
Then $$\mathfrak{M}^{inj}_{\textnormal{der}} = (All, \class{E}, \class{F})$$ is a cofibrantly generated abelian model structure on the exact category $\ch_{pur}$. Its homotopy category is $\class{D}(R)$, the usual derived category of $R$.
\end{corollary}
 
 \begin{proof}
Taking $\class{C} = \{S^0(R)\}$ in Theorem~\ref{them-C-acyclic-models}, then ${}_\class{C}\class{W}$ becomes the class of all $\{S^0(R)\}$-acyclic complexes, which coincides with the class $\class{E}$ of all exact complexes. (If one wants $\class{C}$ be closed under suspensions, then simply take $\class{C} = \{S^n(R)\}$, where $n$ ranges through $\Z$. Then still ${}_\class{C}\class{W} =\class{E}.$) 
It is easy to check that $\class{F} = \rightperp{\class{E}}$, in $\ch_{pur}$. Indeed to see $\rightperp{\class{E}}\subseteq \class{F}$, let us suppose $X \in \rightperp{\class{E}}$ and $M$ is any $R$-module. Then we have $$0 = \Ext^1_{pur}(D^n(M),X) \cong \Ext^1_{\class{P}ur}(M,X_n)$$
by~\cite[Lemma~4.2(1)]{gillespie-G-derived}. This means each $X_n$ is pure-injective. One can then use~\cite[Lemma~4.3]{gillespie-G-derived} to conclude that  $\homcomplex(E,X)$ is exact for all exact complexes $E$. Therefore,  $X \in \class{F}$. Similar arguments will prove the reverse containment, $\class{F} \subseteq \rightperp{\class{E}}$.
So now we see that the result follows from Theorem~\ref{them-C-acyclic-models}.
 \end{proof}
 
 
 \section{The K-flat model for the derived category, recollement, and compact generation}
 
 We now show that we have a K-flat model structure for the usual derived category $\class{D}(R)$. This model structure is a sort of cousin to the one from~\cite{gillespie}, and it too is a monoidal model structure.
 
Again we let $\class{KF}$ denote the class of all K-flat complexes and $\dwclass{PI}$ the class of all chain complexes that are pure-injective in each degree. We let $\class{E}$ denote the class of all exact (acyclic) complexes. Recall that $\ch_{pur}$ denotes the category of all chain complexes of $R$-modules along with the degreewise pure exact structure. 
 
 \begin{theorem}\label{theorem-derived-cat}
 Let $R$ be any ring and $\class{D}(R)$ denote its derived category. 
 There is a cofibrantly generated abelian model structure on the exact category $\ch_{pur}$ represented by the Hovey triple $$\mathfrak{M}^{flat}_{\textnormal{der}} = (\class{KF},\class{E},\dwclass{PI}).$$
We have triangulated category equivalences
$$\class{D}(R) \cong \textnormal{Ho}(\mathfrak{M}^{flat}_{\textnormal{der}}) \cong K(\class{PI}) \cap K\textnormal{-Flat},$$
that is, $\class{D}(R)$ is equivalent to the chain homotopy category of all K-flat complexes with pure-injective components.

Assuming $R$ is commutative, the model structure is monoidal with respect to the usual tensor product of chain complexes. 
 \end{theorem}
 
\begin{proof}
It follows immediately from Proposition~\ref{proposition-pure-acyclic-complexes}(3), Theorem~\ref{them-stovicek-pure} and Theorem~\ref{them-emman-K-flat} that we have the stated abelian model structure. It is cofibrantly generated because each cotorsion pair is cogenerated by a set. It is a model for the derived category because the trivial objects are precisely the acyclic complexes.

As pointed out in Section~\ref{subsection-cot-model}, for abelian categories we need not distinguish between an exact model structure in the sense of~\cite{gillespie-exact model structures} and an abelian model structure compatible with a proper class of short exact sequences as in~\cite{hovey}. So we use~\cite[Theorem~7.2]{hovey} to show that the model structure is monoidal, for a commutative ring $R$. We need to check conditions (a), (b), (c), and (d) from that theorem. 

Condition (a) translates to the statement that (for any chain complex $X$) if we apply $X \otimes_R -$ to a degreewise pure exact sequence then we still have a degreewise pure exact sequence. As shown in the proof of Proposition~\ref{prop-C-acyclic-properties}, it will definitely result in a short exact sequence. In fact, in each degree it is a direct sum of pure short exact sequences (of tensor products) of $R$-modules. Indeed for any $R$-module $M$, the functor $M \otimes_R -$ preserves pure exact sequences, by associativity of tensor products. Moreover, pure exact sequences are closed under direct sums. So condition (a) follows. 

Condition (b) translates to the statement that if $X$ and $Y$ are each K-flat, then $X \otimes_R Y$ too is K-flat. It follows right from the definition of K-flat and associativity of the tensor product. 

Condition (c) translates to the statement that  if $X$ is K-flat and $Y$ pure acyclic, then $X \otimes_R Y$ is also pure acyclic. This too is easily verified with Proposition~\ref{proposition-pure-acyclic-complexes}.

Finally, condition (d) translates to the statement that the unit $S^0(R)$ is a K-flat complex. Clearly it is. This completes the proof that the model structure is monoidal.
\end{proof}
 
 It follows from~\cite[Prop.~1.5.3]{becker} that the above model structure for $\class{D}(R)$ is the right Bousfield localization of Stovicek's model $\mathfrak{M}_{\textnormal{pur}}$ by the model $\mathfrak{M}_{\class{KF}}$. In fact, we have the following recollement. 

\begin{theorem}\label{theorem-recollement}
We have a recollement of compactly generated categories
\[
\xy
(-28,0)*+{\class{D}_{\textnormal{K-flat}}(R)};
(0,0)*+{\class{D}_{pur}(R)};
(25,0)*+{\class{D}(R)};
{(-19,0) \ar (-10,0)};
{(-10,0) \ar@<0.5em> (-19,0)};
{(-10,0) \ar@<-0.5em> (-19,0)};
{(10,0) \ar (19,0)};
{(19,0) \ar@<0.5em> (10,0)};
{(19,0) \ar@<-0.5em> (10,0)};
\endxy
\] which when restricting the first two categories to fibrant objects becomes 
\[
\xy
(-28,0)*+{K_{ac}(\class{PI})};
(0,0)*+{K(\class{PI})};
(25,0)*+{\class{D}(R)};
{(-19,0) \ar (-10,0)};
{(-10,0) \ar@<0.5em> (-19,0)};
{(-10,0) \ar@<-0.5em> (-19,0)};
{(10,0) \ar (19,0)};
{(19,0) \ar@<0.5em> (10,0)};
{(19,0) \ar@<-0.5em> (10,0)};
\endxy
.\] 
\end{theorem}

\begin{proof}
The recollement follows from~\cite[Theorem~4.6]{gillespie-recollement}:  We have three injective cotorsion pairs $$\mathfrak{M}_{pur} = (\class{A}_{pur},\dwclass{PI}) \ , \ \ \ \ \mathfrak{M}_{\class{KF}} = (\class{KF},\exclass{PI}) \ , \ \ \ \ \mathfrak{M}_{\class{E}} = (\class{E}, \class{F}),$$
coming respectively from Theorem~\ref{them-stovicek-pure}, Theorem~\ref{them-emman-K-flat}, and Corollary~\ref{cor-DG-pure-injective-model}.
 They satisfy the hypotheses of~\cite[Theorem~4.6]{gillespie-recollement} so we have a recollement. 

Let us see why all three triangulated homotopy categories are compactly generated. First, it is well-known that $\class{D}(R)$ is compactly generated. Perhaps it is less well-known, but $\class{D}_{pur}(R)$ too is compactly generated. Both of these statements can be seen to be special cases of~\cite[Corollary~4.7(7)]{gillespie-G-derived} ($\class{D}(R)$ is the derived category with respect to the generator $G = R$, while $\class{D}_{pur}(R)$ is the derived category with respect to a representative set $\{G_i\}$ of \emph{all} finitely presented $R$-modules $G_i$.)

It follows that $\class{D}_{\textnormal{K-flat}}(R)$ is compactly generated. Indeed it is enough to show that the equivalent $K_{ac}(\class{PI})$ is compactly generated. But we now have that the inclusion $I : K_{ac}(\class{PI}) \xrightarrow{} K(\class{PI})$ has both a left and a right adjoint. In particular, note that $I$ preserves coproducts and that $K(\class{PI})$ is compactly generated, as it is equivalent to $\class{D}_{pur}(R)$. Now it is an exercise to check that the left adjoint of $I$ carries a set of compact weak generators for $K(\class{PI})$ to a set of compact weak generators for $K_{ac}(\class{PI})$.
\end{proof} 
 
 \section*{Acknowledgements}
 
 The author would like to express his sincere gratitude to Ioannis Emmanouil for carefully reading this manuscript. His comments and suggestions greatly improved the article and clarified some concepts for the author.

\end{document}